\documentclass{amsart}

\usepackage{amsfonts}
\usepackage{amssymb}
\usepackage{times}
\usepackage{xcolor}
\usepackage{tkz-graph}
\usepackage{url}
\usepackage{float}
\usepackage{longtable}
\usepackage{comment}
\usepackage{scalerel}
\usepackage{tasks}
\usepackage[shortlabels]{enumitem}

\NewTasksEnvironment[label =\arabic*, label-width=0.9em,label-format=\bfseries, mathmode]{tabbedEnum}[\item](2)

\usepackage{amsfonts}
\usepackage{amsmath}
\usepackage{eurosym}
\usepackage{geometry}

\newcommand{\red}{\color{red}}

\newcommand{\black}{\color{black}}

\usepackage{caption,booktabs}

\newtheorem{theorem}{Theorem}

\newtheorem{corollary}[theorem]{Corollary}

\newtheorem{definition}[theorem]{Definition}

\newtheorem{lemma}[theorem]{Lemma}

\newtheorem{proposition}[theorem]{Proposition}
\newtheorem{remark}[theorem]{Remark}

\newtheorem*{mains}{Theorem}

\begin{document}

\title{On the simple transposed Poisson algebras and Jordan superalgebras}

\thanks{The author was supported by the Spanish Government through the Ministry of Universities grant `Margarita Salas', funded by the European Union - NextGenerationEU, and by the Centre for Mathematics of the University of Coimbra - UIDB/00324/2020, funded by the Portuguese Government through FCT/MCTES}

\author[Amir Fernández Ouaridi]{Amir Fernández Ouaridi}
\address{Amir Fernández Ouaridi. \newline \indent University of Cádiz, Department of Mathematics, Puerto Real (Spain).
\newline \indent University of Coimbra, CMUC, Department of Mathematics, Apartado 3008,
EC Santa Cruz,
3001-501 Coimbra
(Portugal).}
\email{{\tt amir.fernandez.ouaridi@gmail.com}}

%%%%%%%%%%%%%%%%%%%%%%%%%%%%%%%%%%%%%%%%%%%%%%%%%%%%%%%%%%%%%%%%%%%%%

\thispagestyle{empty}

\begin{abstract}

We prove that a transposed Poisson algebra is simple if and only if its associated Lie bracket is simple. Consequently, any simple finite-dimensional transposed Poisson algebra over an algebraically closed field of characteristic zero is trivial. Similar results are obtained for transposed Poisson superalgebras.
Furthermore, we show that the Kantor double of a transposed Poisson algebra is a Jordan superalgebra. Additionally, a simplicity criterion for the Kantor double of a transposed Poisson algebra is obtained. 

\bigskip

{\it 2020MSC}: 17A70, 17B63, 17A36.

\end{abstract}

\maketitle

\section{Introduction}

In the last years, the study of Poisson algebras has led to various related Poisson type algebraic structures, including generic Poisson algebras, algebras of Jordan brackets, Lie-Yamaguti algebras, Gerstenhaber algebras, Novikov-Poisson algebras, among many others. 
In the paper \cite{tpa}, a dual class of the Poisson algebras was introduced, the transposed Poisson algebras, by changing the roles of the two multiplications in the Leibniz rule. Precisely, a transposed Poisson algebra is a vector space $\mathcal{P}$ endowed with two operations: an associative commutative multiplication $-\circ-$ and an associated Lie bracket $\left\{\cdot, \cdot\right\}$. Additionally, these two operations are required to satisfy the (transposed) Leibniz rule, that is, for any $x,y,z\in \mathcal{P}$ we have:
\begin{equation}\label{tpaid}
    2 x \circ \left\{y, z\right\} = \left\{x \circ y, z\right\} + \left\{y, x \circ z\right\}. 
\end{equation}
The authors show that transposed Poisson algebras share some common properties with Poisson algebras, including the closure undertaking tensor products and the Koszul self-duality as an operad. Since then, the interest in this class has been increasing. The transposed Leibniz rule was realized in \cite{FKL} as the left multiplication of the associative commutative algebra is a $\frac{1}{2}$-derivation of the Lie bracket and this realization was fundamental on the classification of low-dimensional transposed Poisson algebras \cite{bfk23} or on the generalization of the notion to the $n$-ary case \cite{bfk22}. Recall that a $\frac{1}{2}$-derivation of a non-associative algebra $(\mathcal{A}, \cdot)$ is a linear map $D$ in $\mathcal{A}$ such that $2 D(x\cdot y) = D(x)\cdot y + x\cdot D(y)$.
%or in the description of the transposed Poisson structures on Witt type algebras \cite{kk23}. 
Likewise, the notion of transposed Poisson superalgebra has been introduced in the usual way, with the $\mathbb{Z}_2$-graded version of the Leibniz rule being equivalent to the left multiplication of the associative commutative algebra being a $\frac{1}{2}$-superderivation of the Lie bracket.  

In this paper, we focus our interest on simple transposed Poisson (super) algebras. Kac \cite{kac77} used the classification of the simple Lie superalgebras and the TKK functor for Jordan superalgebras in order to classify all the simple finite-dimensional Jordan superalgebras over an algebraically closed field of characteristic zero. Later, Kantor \cite{kantor92} introduced an invertible way to construct a Jordan superalgebra from a Poisson algebra (the Kantor double), this construction preserves the simplicity in both directions, so a classification of the simple finite-dimensional Poisson algebras over an algebraically closed field of characteristic zero was obtained. Also, one of the families of known infinite dimensional simple Jordan superalgebras is precisely the family of those superalgebras that are obtained by the Kantor double \cite{cantarini}.

Recently, it was proven that if a Poisson algebra $(\mathcal{P}, \circ, \left\{\cdot, \cdot\right\})$ is simple, then the Lie algebra $\left\{\mathcal{P}, \mathcal{P}\right\}/(\left\{\mathcal{P}, \mathcal{P}\right\}\cap Z)$, where $Z$ denotes the center of $(\mathcal{P}, \left\{\cdot, \cdot\right\})$, is simple \cite{aa19}. The ``transposed" version of this result does not hold, namely, in a simple transposed Poisson algebra $\mathcal{P}$, the algebra $\mathcal{P}\circ \mathcal{P}/(\mathcal{P}\circ\mathcal{P}\cap Z)$ , where $Z$ is the center of $(\mathcal{P}, \circ)$, is not always simple. In fact, the Witt algebra together with the Laurent polynomials is a simple transposed Poisson algebra, the algebra of Laurent polynomials is perfect and centerless, but not simple. However, a stronger result holds. It turns out that simple transposed Poisson (super) algebras are those that arise on simple Lie (super) algebras, if we omit the trivial case in which the (super) Lie bracket is abelian. Our result is the following.
\begin{mains}
    A transposed Poisson (super) algebra is simple if and only if the associated (super) Lie bracket is simple. 
\end{mains}
As a consequence, the classification of the simple finite-dimensional transposed Poisson algebras and superalgebras over an algebraically closed field of characteristic zero follows thanks to the previous studies of the $\frac{1}{2}$-(super)derivations of the simple Lie (super) algebras \cite{FKL, fili98}. Indeed, it turns out that there are no non-trivial simple finite-dimensional transposed Poisson (super) algebras. In addition, we study the impact of one of the multiplications being perfect on the other one, and we prove that the radical of the Lie algebra is an ideal of the transposed Poisson algebra.

\medskip

Furthermore, we are considering the Kantor double of a transposed Poisson algebra $\mathfrak{J}(\mathcal{P})$. Let us recall this construction briefly.
Following Kantor's notation, given a vector space $\mathbb{V}$, we denote by $\mathbb{V}^s$ the vector space which is a copy of $\mathbb{V}$, but with odd parity.
Now, suppose $(\mathcal{P}, \circ, \left\{\cdot, \cdot\right\})$ 
is an algebra equipped with two multiplications: the first one is associative and commutative and the second one is skew-symmetric. This class of algebras are called dot-bracket algebras \cite{km95}. Consider the algebra $\mathfrak{J}(\mathcal{P})$ with underlying vector space $\mathcal{P} \oplus \mathcal{P}^s$ and with the operation $*$ given for $a, b \in \mathcal{P}$ and corresponding $a^s, b^s \in \mathcal{P}^s$ by the following relations
\begin{equation*}
    a*b = a\circ b, \quad a^s * b = a * b^s = (a\circ b)^s, \quad a^s * b^s = \left\{a,b\right\}.
\end{equation*}
Thus, $\mathfrak{J}(\mathcal{P}) = \mathcal{P} \oplus \mathcal{P}^s$ is a superalgebra with double the dimension of $\mathcal{P}$, which is called the Kantor double of $\mathcal{P}$. A similar construction, can be considered if $\mathcal{P}$ is a dot-bracket superalgebra \cite{km95}. Kantor proved that the double of a Poisson (super) algebra is a Jordan superalgebra \cite{kantor92}, but there exist other non-Poisson (super) algebras whose Kantor doubles are Jordan superalgebras.  Recall that a Jordan superalgebra is a superalgebra $(\mathcal{J} := \mathcal{J}_0 \oplus\mathcal{J}_1, \cdot)$ satisfying the supercommutativity, $x \cdot y = (-1)^{xy} y \cdot x$, and the Jordan superidentity
\begin{equation} \label{sjord}
    (-1)^{xz}[L_{x\cdot y}, L_z] + (-1)^{yx}[L_{y\cdot z}, L_x] + (-1)^{zy}[L_{z\cdot x}, L_y] = 0.
\end{equation}
for any homogeneous elements $x, y, z\in \mathcal{J}_0 \cup \mathcal{J}_1$ and where $L_x \in End(\mathcal{J})$ denotes the linear operator of left multiplication by $x\in \mathcal{J}$ and the bracket of operators is $[L_{x}, L_{y}] = L_{x} L_{y} - (-1)^{xy} L_{y} L_{x}$.

The algebras that produce a Jordan superalgebra through the Kantor double are usually called {Jordan brackets} \cite{mz19}. In the case in which $(\mathcal{P}, \circ)$ is unital, Jordan brackets are characterized by a set of pseudoidentities \cite{km95}, involving the distinguished derivation $D(x) =\left\{x, 1\right\}$. Later, it was shown that in the unital case, Jordan brackets are in one to one correspondence with contact brackets \cite{cantarini}, also known as generalized Poisson algebras. Unital transposed Poisson algebras are contact brackets and Jordan brackets, but for non-unital algebras they are not necessarily. In fact, one can prove that for unital transposed Poisson algebras \cite{bfk23}, the associated Lie algebra is given by $\left\{x, y\right\}= D(x)\circ y - x\circ D(y)$, where $D(x) =\left\{x, 1\right\}$. { Moreover, the Kantor double of the commutator of a Novikov-Poisson algebra is a Jordan superalgebra \cite{z14}. Recall that if we take the commutator of the Novikov multiplication of a Novikov-Poisson algebra, we obtain a transposed Poisson algebra \cite{tpa}. Whether every transposed Poisson algebra can be realized in this manner remains uncertain. In any case, finding new classes of Jordan brackets is an interesting issue in Jordan algebras theory, and it turns out that transposed Poisson algebras are Jordan brackets.  }
\begin{mains}
    The Kantor double $\mathfrak{J}(\mathcal{P})$ of a transposed Poisson algebra $\mathcal{P}$ is a Jordan superalgebra, i.e., transposed Poisson algebras are Jordan brackets. Moreover, the algebra $\mathfrak{J}(\mathcal{P})$ is simple if and only if $\mathcal{P}$ is simple and $\mathcal{P}\circ\mathcal{P} = \mathcal{P}$.
\end{mains}

{Lastly, we show that  simple transposed Poisson algebras with a perfect associative commutative multiplication are unital. Note that this is similar to what happens with Novikov-Poisson algebras (see \cite[Lemma 7]{z14}). Thus, there are no simple Jordan superalgebras that can be constructed through the Kantor double from non-unital transposed Poisson algebras. }

\medskip

Any vector space, algebra or homomorphism in this paper will be over an arbitrary base field $\mathbb{F}$ of characteristic not two, if nothing is specified. We will denote the associative commutative multiplication $\circ$ by concatenation, and we will assume that it is not necessarily unital.

\section{On the simple transposed Poisson algebras}\label{sec2}

In this section, we assume the Lie bracket is not abelian for any transposed Poisson algebra. An ideal in a transposed Poisson algebra $\mathcal{P}$ is a proper subspace $\mathcal{I}$ such that $\mathcal{I}\mathcal{P} \subset \mathcal{I}$ and $\left\{\mathcal{I}, \mathcal{P}\right\}\subset \mathcal{I}$. We say $\mathcal{P}$ is simple if it contains no ideals. Given a transposed Poisson algebra $\mathcal{P}$, if the Lie algebra $(\mathcal{P}, \left\{\cdot, \cdot\right\})$ is not perfect (that is, $\mathcal{P}^2_{\left\{\cdot, \cdot\right\}} := \left\{\mathcal{P}, \mathcal{P}\right\} \neq \mathcal{P}$), then by Leibniz identity (\ref{tpaid}) we have that $\mathcal{P}^2_{\left\{\cdot, \cdot\right\}}$ is an ideal of the associative commutative algebra $(\mathcal{P}, \circ)$, and so is an ideal of $\mathcal{P}$. Therefore, in any simple transposed Poisson algebra the Lie bracket must be perfect. Moreover, let us introduce the following notion.

\begin{definition}
    A transposed quasi-ideal of a transposed Poisson algebra $(\mathcal{P}, \circ, \left\{\cdot,\cdot\right\})$ is a proper subspace $\mathcal{I}$ of $\mathcal{P}$ such that
    $\left\{\mathcal{P}, \mathcal{I}\right\}\subset \mathcal{I}$ and $\left\{\mathcal{P}\mathcal{I}, \mathcal{P}\right\}\subset \mathcal{I}$.
\end{definition}

This notion is the transposed version of the notion of a quasi-ideal of a Poisson algebra. Recall that a quasi-ideal is a proper subspace $\mathcal{I}$ of $\mathcal{P}$ such that $\mathcal{P}\mathcal{I}\subset \mathcal{I}$ and $\left\{\mathcal{P}, \mathcal{I}\right\} \mathcal{P}\subset \mathcal{I}$. 
Any simple Poisson algebra contains no quasi-ideals and the same is valid for transposed Poisson algebras and transposed quasi-ideals as we show in the next result.

\begin{lemma}\label{quasi}
    A simple transposed Poisson algebra contains no transposed quasi-ideals.
\end{lemma}
\begin{proof}
    Suppose $\mathcal{I}$ is a transposed quasi-ideal of $\mathcal{P}$. Consider a maximal subspace $\mathcal{I}'$ such that $\left\{\mathcal{P}, \mathcal{I}'\right\} \subset \mathcal{I}$. We will show that $\mathcal{I}'$ is an ideal of $\mathcal{P}$. Observe that the Lie bracket is perfect, because $\mathcal{P}$ is simple, so $\mathcal{I}'\neq \mathcal{P}$. Also, we can assume  $\mathcal{P}\mathcal{I} \subset \mathcal{I}'$, by the maximality of $\mathcal{I}'$ and since  $\mathcal{I}$ is a transposed quasi-ideal.
    Now, since $x[y,z] + y[z,x] + z[x,y] = 0$ (see \cite[Theorem 2.5]{tpa}), we have that 
    $\mathcal{I}' \left\{y, z\right\} \subset y\left\{z, \mathcal{I}'\right\} + z\left\{\mathcal{I}', y\right\} \subset \mathcal{I}'$. Finally, the Lie bracket is perfect so $\mathcal{I}' \mathcal{P} \subset \mathcal{I}'$, and also $\left\{\mathcal{P}, \mathcal{I}'\right\} \subset \mathcal{I} \subset \mathcal{I}'$. Therefore, $\mathcal{I}'$ is a proper ideal, which contradicts the simplicity of $\mathcal{P}$.
\end{proof}

\begin{lemma}\label{ideal}
    Let $(\mathcal{P}, \circ, \left\{\cdot,\cdot\right\})$ be a transposed Poisson algebra and suppose that the associated Lie bracket is perfect, i.e., $\mathcal{P}^2_{\left\{\cdot, \cdot\right\}} =\mathcal{P}$. Then any ideal in the Lie algebra $(\mathcal{P}, \left\{\cdot,\cdot\right\})$ is a transposed quasi-ideal.
\end{lemma}
\begin{proof}
    Suppose that $\mathcal{I}$ is an ideal of the associated Lie bracket of a transposed Poisson algebra $\mathcal{P}$. By Theorem 2.5 in \cite{tpa}, we have the identity $\left\{hx, \left\{y, z\right\}\right\} = - \left\{hy, \left\{z, x\right\}\right\} - \left\{hz, \left\{x, y\right\}\right\}$, so we can write
    $$\left\{\mathcal{P}\mathcal{I},\mathcal{P}\right\} = \left\{\mathcal{P}\mathcal{I}, \left\{\mathcal{P}, \mathcal{P}\right\}\right\} \subset  \left\{\mathcal{P} \mathcal{P}, \left\{\mathcal{P}, \mathcal{I}\right\}\right\} + \left\{\mathcal{P}\mathcal{P}, \left\{\mathcal{I}, \mathcal{P}\right\}\right\} \subset \mathcal{I}.$$
    Hence, $\mathcal{I}$ is a transposed quasi-ideal.
\end{proof}

As a consequence of the previous two lemmas, we have the following result that shows that a transposed Poisson algebra is simple if and only the Lie bracket is simple.

\begin{theorem}\label{main1}
    Any simple transposed Poisson algebra has simple Lie bracket. 
\end{theorem}
\begin{proof}
    Suppose $\mathcal{P}$ is a simple transposed Poisson algebra, then the Lie bracket $(\mathcal{P}, \left\{\cdot, \cdot\right\})$ is perfect. Now, if we suppose that $(\mathcal{P}, \left\{\cdot, \cdot\right\})$ is not simple, then any ideal is a quasi-ideal by Lemma \ref{ideal}, but this contradicts Lemma \ref{quasi}, because $\mathcal{P}$ is simple. Therefore, the Lie bracket must be simple.
\end{proof}

The classification of simple finite-dimensional transposed Poisson algebras follows. 

\begin{theorem}\label{clasi1}
    Suppose that $\mathbb{F}$ is algebraically closed and $\textrm{char}(\mathbb{F}) = 0$, then any simple finite-dimensional transposed Poisson algebra is trivial. 
\end{theorem}
\begin{proof}
    By Theorem \ref{main1}, any simple transposed Poisson algebra is defined on a simple Lie bracket and if $\mathcal{P}$ is finite dimensional, then any $\frac{1}{2}$-derivation is trivial \cite{fili98} (i.e. a multiplication by a field element), so any transposed Poisson algebra has trivial associative commutative multiplication (see \cite[Theorem 7]{FKL}).
\end{proof}

We can conclude the following about a finite-dimensional transposed Poisson algebra $\mathcal{P}$ over an  algebraically closed field of characteristic zero, by looking at its Lie bracket. There are four possibilities summarized below.

\begin{enumerate}
    \item If the Lie bracket is simple, then the associative commutative multiplication is trivial.

    \item If the Lie bracket is perfect and non-simple, then there exist a transposed quasi-ideal, and, therefore, an ideal. We will study this case in the next section and prove that the associative part must be nilpotent.
    
    \item If the Lie bracket is non-perfect and non-abelian, then $\mathcal{P}^2_{\left\{\cdot, \cdot\right\}}$ is an ideal of $\mathcal{P}$.

    \item If the Lie bracket is abelian, we have an associative commutative algebra. Obviously, the transposed Poisson algebra is simple if the associative commutative multiplication is simple.
\end{enumerate}

\begin{remark}
    Observe that there are simple infinite dimensional Lie algebras that admit a non-trivial transposed Poisson structure over the complex field, such as the Witt algebra (simple), with the algebra of Laurent polynomials (non-simple).
\end{remark}

The following example shows that there are non-trivial simple finite-dimensional transposed Poisson algebras over fields of prime characteristic.

\begin{remark}\label{example}
    Consider the special linear Lie algebra $\mathfrak{sl}_2(\mathbb{F})$, where $\mathbb{F}$ is a field of characteristic three. For some basis $e_1, e_2, e_3$, the algebra $\mathfrak{sl}_2(\mathbb{F})$ is given by the products
    $$\left\{e_1, e_2\right\} = e_3, \quad \left\{e_3, e_2\right\} = - 2 e_2, \quad \left\{e_3, e_1\right\} = 2 e_1.$$
    In characteristic three, it has non-trivial $\frac{1}{2}$-derivations. In fact, a $\frac{1}{2}$-derivation $D$ of $\mathfrak{sl}_2(\mathbb{F})$ is of the form
    $$D(e_1)=\lambda e_1 + \mu e_2, \quad D(e_2)=\nu e_1 + \lambda e_2, \quad D(e_3)= \lambda e_3,$$
     for arbitrary scalars $\lambda, \mu, \nu \in \mathbb{F}$. Following the arguments of \cite{bfk23}, one can construct the transposed Poisson structures on it, and they are precisely of the form 
     $${\mathcal{P}}^{\alpha, \beta}:\left\{ 
\begin{tabular}{l}
$ e_1 \circ e_1 = \alpha e_2, \quad e_2 \circ e_2 = \beta e_1,$ \\ 
$\left\{e_1, e_2\right\} = e_3, \quad \left\{e_3, e_2\right\} = - 2 e_2, \quad \left\{e_3, e_1\right\} = 2 e_1,$
\end{tabular}%
\right. $$
where $\alpha, \beta \in \mathbb{F}$ are such that $\alpha\beta =0$. Lastly, note that it is not Poisson and that since $\mathfrak{sl}_2(\mathbb{F})$ is simple, ${\mathcal{P}}^{\alpha, \beta}$ is also simple. 
\end{remark}

Another consequence of Theorem \ref{main1} is the next corollary.
 
\begin{corollary}
     A non-trivial transposed Poisson structure defined on a simple associative commutative algebra { (if it exists)} has simple associated Lie bracket.
\end{corollary}

Note that there are no known examples of transposed Poisson algebras in which both multiplications are simple. 

\begin{theorem}\label{dsum}
        Let $(\mathcal{P}, \circ, \left\{\cdot,\cdot\right\})$ be a transposed Poisson algebra. The associated Lie algebra $(\mathcal{P}, \left\{\cdot, \cdot\right\})$ is a direct sum of simple ideals if and only if the transposed Poisson algebra $\mathcal{P}$ is a direct sum of simple ideals. 
\end{theorem}
\begin{proof}
    Suppose the Lie algebra $(\mathcal{P}, \left\{\cdot, \cdot\right\})$ is a direct sum of the simple ideals $\mathcal{I}_r$ for $r$ in some set of indices $I$, then it is perfect.  Also, each of these ideals is a quasi-ideal, by Lemma \ref{ideal}.  Now, observe that the maximal subspace $\mathcal{I}'_r$ such that $\left\{\mathcal{P}, \mathcal{I}'_r\right\} \subset \mathcal{I}_r$, from the proof of Lemma \ref{quasi}, is $\mathcal{I}'_r= \mathcal{I}_r$, because $\left\{\mathcal{P}, \mathcal{I}_k\right\} = \mathcal{I}_k$ for $k\in I$. Hence, by a similar argument, it follows that $\mathcal{I}_r\mathcal{P}\subset\mathcal{I}_r$, so every $\mathcal{I}_r$ is a ideal of the transposed Poisson algebra $\mathcal{P}$. Finally, since the ideals $\mathcal{I}_r$ are simple as Lie algebras, they are simple as transposed Poisson algebras.
    Conversely, it is clear that if the transposed Poisson algebra $\mathcal{P}$ is a direct sum of ideals, then the associated Lie algebra is also a direct sum of these ideals. Now, since they are simple as transposed Poisson algebras, their Lie bracket must be simple, by Theorem \ref{main1}.
\end{proof}

In an algebraically closed field of characteristic zero, finite-dimensional semisimple (radical is zero) Lie algebras are precisely those that can be written as a direct sum of simple ideals. Also, semisimple Lie algebras have no non-trivial $\frac{1}{2}$-derivations \cite{FKL}. Hence, they have no non-trivial transposed Poisson structures. By the previous theorem, we have a more general result than Theorem \ref{clasi1}.

\begin{theorem}
    Suppose that $\mathbb{F}$ is algebraically closed and $\textrm{char}(\mathbb{F}) = 0$. Let $(\mathcal{P}, \circ, \left\{\cdot,\cdot\right\})$ be a finite-dimensional transposed Poisson algebra. If the transposed Poisson algebra $\mathcal{P}$ is a direct sum of simple ideals, then it is trivial.
\end{theorem}

\section{Transposed Poisson algebras with a perfect multiplication}

Suppose $\mathbb{F}$ is an algebraically closed field of characteristic zero in this section. We have proved that there are no non-trivial simple finite-dimensional transposed Poisson algebras. In this section, we study how one of the multiplications being perfect affects the other one. Let us recall a useful result about generalized derivations of Lie algebras.

A generalized derivation of a Lie algebra $(\mathcal{L}, \left\{\cdot,\cdot\right\})$ is a linear map $D: \mathcal{L}\rightarrow \mathcal{L}$ such that there exists two additional linear maps $D', D'': \mathcal{L}\rightarrow \mathcal{L}$ such that 
$$\left\{D(x), y\right\} + \left\{x, D'(y)\right\} = D''(\left\{x, y\right\}).$$
Leger and Luks proved that the generalized derivations of a finite-dimensional Lie algebra preserve the radical of the algebra \cite[Theorem 6.4]{leger}.  Observe that $\frac{1}{2}$-derivations are generalized derivations, and recall that the left multiplication of the associative commutative operation in any transposed Poisson algebra constructed on $\mathcal{L}$ is a $\frac{1}{2}$-derivation of $\mathcal{L}$. Let us denote the descending derived series of an ideal as  $\mathcal{I}^{(0)} = \mathcal{I}$ and $\mathcal{I}^{(k)} = \left\{\mathcal{I}^{(k-1)}, \mathcal{I}^{(k-1)} \right\}$ for $k\geq 1$.
The next result follows. 

\begin{theorem}
    Let $(\mathcal{P}, \circ, \left\{\cdot,\cdot\right\})$ be a finite-dimensional transposed Poisson algebra. Denote by $\mathcal{R}$ the radical of the associated Lie algebra $(\mathcal{P}, \left\{\cdot,\cdot\right\})$, then $\mathcal{R}^{(i)}$ is a (not necessarily proper) ideal of the transposed Poisson algebra $\mathcal{P}$, for $i\geq 0$. 
\end{theorem}
\begin{proof}
    The radical $\mathcal{R}$ of the associated Lie algebra is an ideal of $\mathcal{P}$, as we have shown above. Now, for the second part, we note that for an arbitrary ideal $\mathcal{I}$ of a transposed Poisson algebra $\mathcal{P}$ we have that
    $$ \mathcal{P} \left\{\mathcal{I}, \mathcal{I}\right\}\subset  \left\{\mathcal{P}\mathcal{I}, \mathcal{I}\right\} +\left\{\mathcal{I}, \mathcal{P}\mathcal{I}\right\} \subset  \left\{\mathcal{I}, \mathcal{I}\right\}, \quad \left\{\mathcal{P}, \left\{\mathcal{I}, \mathcal{I}\right\}\right\}\subset \left\{\mathcal{I}, \left\{\mathcal{P}, \mathcal{I}\right\}\right\} + \left\{\mathcal{I}, \left\{\mathcal{I}, \mathcal{P}\right\}\right\}\subset  \left\{\mathcal{I}, \mathcal{I}\right\}.$$
    Hence, the subspace $\left\{\mathcal{I}, \mathcal{I}\right\}$ is an ideal of $\mathcal{P}$. In particular, we conclude that $\mathcal{R}^{(i)}$  is an ideal of $\mathcal{P}$.
\end{proof}

Observe that the previous theorem does not hold in prime characteristic and a counterexample can be found in Remark \ref{example}. 

\begin{corollary}
        Let $(\mathcal{P}, \circ, \left\{\cdot,\cdot\right\})$ be a finite-dimensional transposed Poisson algebra. Then $\mathcal{P}\mathcal{P}\subset \mathcal{R}$ as subespaces, where $\mathcal{R}$ is the radical of the associated Lie algebra.
\end{corollary}
\begin{proof}
    If $\mathcal{R}= \mathcal{P}$, this is clear. Suppose $\mathcal{R}\neq \mathcal{P}$, then the algebra $(\mathcal{P}/\mathcal{R}, \left\{\cdot,\cdot\right\})$ is semisimple, hence $(\mathcal{P}/\mathcal{R}, \circ, \left\{\cdot,\cdot\right\})$ is trivial and $(\mathcal{P}/\mathcal{R}, \circ)$ is the zero algebra. Therefore, it follows that $\mathcal{P}\mathcal{P}\subset \mathcal{R}$.
\end{proof}

The consequence is the next result about transposed Poisson algebras with perfect associative commutative part.

\begin{corollary}\label{perfectac}
    Let $(\mathcal{P}, \circ, \left\{\cdot,\cdot\right\})$ be a finite-dimensional transposed Poisson algebra such that $\mathcal{P}\mathcal{P} = \mathcal{P}$, then the associated Lie algebra is solvable. In particular, if the associative commutative algebra is unital, then the associated Lie algebra is solvable.
\end{corollary}

Note that the example in Remark \ref{examplelie} shows that $\mathcal{P}\mathcal{P}$ is not always an ideal of the Lie part. Also, the next remark shows that the solvability can not be replaced by nilpotency in Corollary \ref{perfectac}.

\begin{remark}
    The complex Lie algebra with a basis $e_1, e_2, e_3$ given by the products $\left\{e_1, e_3\right\}=e_1, \left\{e_2, e_3\right\} = e_2$ is non-nilpotent solvable, and together with the unital associative commutative algebra given by $e_1e_3=e_1, e_2e_3=e_2, e_3e_3 = e_3$, they form a transposed Poisson structure.
\end{remark}

On the other hand, if the Lie part is perfect, we have the following result.

\begin{corollary}
    Let $(\mathcal{P}, \circ, \left\{\cdot,\cdot\right\})$ be a finite-dimensional transposed Poisson algebra such that $\left\{\mathcal{P}, \mathcal{P}\right\} = \mathcal{P}$, then the associative commutative algebra $(\mathcal{P}, \circ)$ is nilpotent.
\end{corollary}
\begin{proof}
    Let us show that $\mathcal{P}^{2^n+1}\subset \mathcal{R}^{(n)}$ for $n\geq1$. We proceed by induction using the equation (\ref{propeq5}). For $n=1$, we have  that $\mathcal{P}^{3} = \mathcal{P}\mathcal{P}\left\{\mathcal{P}, \mathcal{P}\right\} \subset \left\{\mathcal{P}\mathcal{P}, \mathcal{P}\mathcal{P}\right\} \subset \left\{\mathcal{R}, \mathcal{R}\right\} =\mathcal{R}^{(1)} $. For $n>1$, we have 
    $$\mathcal{P}^{2^n+1} = \mathcal{P}^{2^{n-1}}\mathcal{P}^{2^{n-1}}\mathcal{P} = \mathcal{P}^{2^{n-1}}\mathcal{P}^{2^{n-1}}\left\{\mathcal{P}, \mathcal{P}\right\}\subset \left\{\mathcal{P} \mathcal{P}^{2^{n-1}}, \mathcal{P} \mathcal{P}^{2^{n-1}}\right\}\subset \left\{\mathcal{R}^{(n-1)}, \mathcal{R}^{(n-1)}\right\}=\mathcal{R}^{(n)}.$$
    Now, since $\mathcal{R}$ is solvable, there exist some $k$ such that $\mathcal{P}^{2^k+1}\subset \mathcal{R}^{(k)} = 0$. Therefore, $(\mathcal{P}, \circ)$  is nilpotent.
\end{proof}

\section{On the simple transposed Poisson superalgebras}\label{sec3}

Let $\mathbb{F}$ be an arbitrary field of characteristic different from two. In this section, we show that the results in the second section are also valid for superalgebras. Although our  arguments are the same, they deserve a special mention. A transposed Poisson superalgebra is a $\mathbb{Z}_2$-graded vector space $\mathcal{P} = \mathcal{P}_0 \oplus \mathcal{P}_1$ endowed with two multiplications: an associative supercommutative multiplication $-\circ-$ and a Lie superalgebra multiplication $\left\{\cdot, \cdot\right\}$. Recall that the Jacobi superidentity can be written as 
\begin{equation}\label{sjacob}
    (-1)^{|x||z|}\left\{\left\{x, y\right\}, z\right\} + (-1)^{|y||x|}\left\{\left\{y, z\right\}, x\right\} + (-1)^{|z||y|}\left\{\left\{z, x\right\}, y\right\} = 0.
\end{equation}
Additionally, these two operations are required to satisfy the (transposed) Leibniz superidentity:
\begin{equation}\label{supertpaid}
    2 x \circ \left\{y, z\right\} = \left\{x \circ y, z\right\} + (-1)^{|x||y|}\left\{y, x \circ z\right\}, 
\end{equation}
for homogeneous elements $x,y,z\in \mathcal{P}_0 \cup \mathcal{P}_1$. As usual, $|x|$ denotes the parity of $x$. Although, we may write $(-1)^{x}:= (-1)^{|x|}$. For convenience, we will denote the multiplication $\circ$ by concatenation.

\medskip

The identities in \cite[Theorem 2.5]{tpa} can be generalized to the superalgebra case as follows. They can be constructed with the usual rules to construct superidentities from identities and their proof is analogous to the proof for identities in \cite{tpa}.

\begin{proposition}\label{propeqs}
    Let $\mathcal{P} = \mathcal{P}_0 \oplus \mathcal{P}_1$ be a transposed Poisson superalgebra. Then for $x, y, z, h, u, v\in \mathcal{P}_0 \cup \mathcal{P}_1$ we have:
    \begin{equation}\label{propeq1}
            (-1)^{xz} x \left\{ y, z \right\} + (-1)^{yx} y \left\{ z, x \right\} + (-1)^{zy} z \left\{ x, y \right\} = 0.
    \end{equation}
    \begin{equation}\label{propeq2}
        (-1)^{xz} \left\{h \left\{x, y\right\}, z\right\} + (-1)^{yx} \left\{h \left\{y, z\right\}, x\right\} + (-1)^{zy} \left\{h \left\{z, x\right\}, y\right\} = 0.
    \end{equation}
    \begin{equation}\label{propeq3}
        (-1)^{xz} \left\{hx, \left\{y, z\right\}\right\} + (-1)^{yx} \left\{hy, \left\{z, x\right\}\right\} + (-1)^{zy} \left\{hz, \left\{x, y\right\}\right\} = 0.
    \end{equation}
    \begin{equation}\label{propeq4}
        (-1)^{xz} \left\{h, x\right\}\left\{y, z\right\} + (-1)^{yx} \left\{h, y\right\}\left\{z, x\right\} + (-1)^{zy} \left\{h, z\right\}\left\{x, y\right\}= 0.
    \end{equation}
    \begin{equation}\label{propeq5}
        \left\{xu, vy\right\} + (-1)^{uv}\left\{xv, uy\right\} = 2 (-1)^{ux+vx}uv\left\{x, y\right\}.
    \end{equation}
    \begin{equation}\label{propeq6}
        (-1)^{vx + yu}x\left\{u, yv\right\} + (-1)^{vu+vy}v\left\{xy, u\right\} + (-1)^{xu+xy}yu\left\{v, x\right\} = 0.
    \end{equation}
\end{proposition}

Recall that a $\mathbb{Z}_2$-graded ideal in a transposed Poisson superalgebra $\mathcal{P} = \mathcal{P}_0 \oplus \mathcal{P}_1$ is a $\mathbb{Z}_2$-graded vector space $\mathcal{I} = \mathcal{I}_0 \oplus \mathcal{I}_1$ such that $\mathcal{I}\mathcal{P} \subset \mathcal{I}$ and $\left\{\mathcal{I}, \mathcal{P}\right\}\subset \mathcal{I}$. A superalgebra is simple if it contains no graded ideals. If the Lie superalgebra $(\mathcal{P}, \left\{\cdot, \cdot\right\})$ associated to a transposed Poisson superalgebra is not perfect, then by the Leibniz identity (\ref{supertpaid}), $\mathcal{P}^2_{\left\{\cdot, \cdot\right\}}$ is a graded ideal of $\mathcal{P}$. Indeed, we have
$$\mathcal{P} \mathcal{P}^2_{\left\{\cdot, \cdot\right\}}= (\mathcal{P}_0 \oplus \mathcal{P}_1)\left\{\mathcal{P}_0 \oplus \mathcal{P}_1, \mathcal{P}_0 \oplus \mathcal{P}_1\right\}\subset \sum_{i, j, k \in \mathbb{Z}_2} \left\{\mathcal{P}_i \mathcal{P}_j, \mathcal{P}_k\right\} + \sum_{i, j, k \in \mathbb{Z}_2} \left\{\mathcal{P}_j \mathcal{P}_i, \mathcal{P}_k\right\} \subset \mathcal{P}^2_{\left\{\cdot, \cdot\right\}}.$$
Therefore, in any simple transposed Poisson superalgebra the Lie bracket must be perfect. Now, let us introduce the key notion of a graded transposed quasi-ideal.

\begin{definition}
    A graded transposed quasi-ideal of a transposed Poisson superalgebra $\mathcal{P}=\mathcal{P}_0 \oplus \mathcal{P}_1$ is a $\mathbb{Z}_2$-graded proper subspace $\mathcal{I} = \mathcal{I}_0 \oplus \mathcal{I}_1$ of $\mathcal{P}$ such that
    $\left\{\mathcal{P}, \mathcal{I}\right\}\subset \mathcal{I}$ and $\left\{\mathcal{P}\mathcal{I}, \mathcal{P}\right\}\subset \mathcal{I}$.
\end{definition}

The following results are the equivalents of Lemma \ref{quasi} and Lemma \ref{ideal} for superalgebras. The proof is a repetition of the proof of the corresponding lemmas, but considering a graded ideal and using the equations in Proposition \ref{propeqs}. The same applies for Theorem \ref{main2}.

\begin{lemma}\label{quasisuper}
    A simple transposed Poisson superalgebra contains no graded transposed quasi-ideals.
\end{lemma}

\begin{lemma}\label{idealsuper}
    Let $(\mathcal{P}, \circ, \left\{\cdot,\cdot\right\})$ be a transposed Poisson superalgebra and suppose that the associated Lie superalgebra is perfect, i.e., $\mathcal{P}^2_{\left\{\cdot, \cdot\right\}} =\mathcal{P}$. Then any graded ideal in the Lie superalgebra $(\mathcal{P}, \left\{\cdot,\cdot\right\})$ is a graded transposed quasi-ideal.
\end{lemma}

Using the previous lemmas, the equivalent of Theorem \ref{main1} for superalgebras is obtained. 

\begin{theorem}\label{main2}
    Any simple transposed Poisson superalgebra has simple (or abelian) super-Lie bracket. 
\end{theorem}

By a similar argument as in Theorem \ref{clasi1}, but using the fact that simple finite-dimensional Lie superalgebras have no non-trivial transposed Poisson superalgebra structure \cite[Corollary 12]{FKL}, we can conclude the classification theorem.

\begin{theorem}\label{clasi2}
    Suppose that $\mathbb{F}$ is algebraically closed and $\textrm{char}(\mathbb{F}) = 0$, then any simple finite-dimensional transposed Poisson superalgebra is trivial.
\end{theorem}

\section{The Kantor double of a transposed Poisson algebra and superalgebra}

On this section, let $\mathbb{F}$ be an arbitrary field of characteristic different from two. Let us recall the Kantor double construction on the most general context. Let $(\mathcal{P} = \mathcal{P}_0\oplus \mathcal{P}_1, \circ, \left\{\cdot,\cdot\right\})$ be a superalgebra equipped with two multiplications: the first one is associative and supercommutative and the second one is super skew-symmetric. This class of algebras are called dot-bracket superalgebras \cite{km95}. We will assume that they are not necessarily unital, as in the rest of the paper. Denote by $\mathcal{P}^s$ a duplication of the space $\mathcal{P}$ with opposite parity of the elements, so if $x\in\mathcal{P}_i$, its corresponding copy $x^s\in\mathcal{P}^s$ has parity $i+1 \in \mathbb{Z}_2$.  Then let $\mathfrak{J}(\mathcal{P})$ be the vector space $\mathcal{P}\oplus \mathcal{P}^s$ endowed with the multiplication $*$ given for homogeneous elements $x, y \in \mathcal{P}_0\cup \mathcal{P}_1$ by
\begin{equation*}
    x*y = x\circ y, \quad x^s * y = (-1)^x x* y^s =  (x\circ y)^s, \quad x^s * y^s = (-1)^x \left\{x, y\right\}.
\end{equation*}
By this construction, the algebra $\mathfrak{J}(\mathcal{P})$ is a superalgebra with even space $\mathfrak{J}(\mathcal{P})_0 = \mathcal{P}_0 \oplus \mathcal{P}_1^s$ and odd space $\mathfrak{J}(\mathcal{P})_1 = \mathcal{P}_1\oplus \mathcal{P}_0^s$. 
Also, $\mathfrak{J}(\mathcal{P}) = \mathfrak{J}(\mathcal{P})_0 \oplus \mathfrak{J}(\mathcal{P})_1$ has ``double" the dimension of $\mathcal{P}$, hence its name: the Kantor double of $\mathcal{P}$.  

\medskip

We will prove that the Kantor double of a transposed Poisson (super) algebra is a Jordan superalgebra. We may refer to superalgebras simply by algebra if the context is clear. Let $(\mathcal{P}=\mathcal{P}_0 \oplus \mathcal{P}_1, \circ, \left\{\cdot, \cdot\right\})$ be a transposed Poisson superalgebra and denote by $P_a, Q_a: \mathcal{P} \rightarrow \mathcal{P}$ the linear operators corresponding to the left multiplications given by $P_a(x) = ax$ and $Q_a(x) = \left\{a,x\right\}$, for an homogeneous element $a\in \mathcal{P}_0 \cup \mathcal{P}_1$ . Then, we have the following useful relations between these linear operators.

\begin{proposition}
    \label{rels}
Given $\mathcal{P}$ a transposed Poisson superalgebra and $x,y,z\in \mathcal{P}_0 \cup \mathcal{P}_1$, we have the relations:
\begin{enumerate}
    \item The associativity of $\circ$ is equivalent to $[P_x, P_y] = 0$, since $\circ$ is also supercommutative.
    %$$ (x) = P_aP_b (x) - P_bP_a (x) = a(bx) - b(ax) = (ab)x - (ba)x = 0$$

    \item The Jacobi superidentity (\ref{sjacob}) is equivalent to  $[Q_x, Q_y] = Q_{\left\{x, y\right\}}$, due to the super skew-symmetry.

    \item The Leibniz rule (\ref{supertpaid}) can be written as 
    $[P_x, Q_y] = \frac{1}{2}(Q_{xy} - (-1)^{xy}Q_{y}P_{x})$ or %  {\red $Q_{a}P_{b}-Q_{b}P_{a} = 2 P_{\left\{a,b\right\}}.$}
    %$$[P_a, Q_b](x) = P_a Q_b(x) - Q_b P_a(x) = a\left\{b, x\right\} - \left\{b, ax \right\} = \frac{1}{2}(\left\{ab, x \right\}- \left\{b, ax \\})  $$
    $Q_{xy} + (-1)^{xy} Q_{y}P_{x} = 2 P_{x}Q_{y}$.    
    
    \item Moreover, a consequence of this identity is  
    $$(-1)^{xz}Q_{\left\{x, y\right\}z}  + (-1)^{yx}Q_{\left\{y, z\right\}x} + (-1)^{zy}Q_{\left\{z, x\right\}y} =0. $$
\end{enumerate}
Note that the bracket of linear operators is graded, that is, for $A, B \in End(\mathcal{P})$, we have $[A, B] := AB - (-1)^{AB} B A$.
\end{proposition}
\begin{proof}
    The relations (1) and (2) are well-known for associative commutative and Lie superalgebras. The first relation in (3) applied to an element $z \in \mathcal{P}$ gives us
    $ x\left\{y, z\right\} - (-1)^{xy}\left\{y, x z\right\} = \frac{1}{2}\left( \left\{xy, z\right\} - (-1)^{xy}\left\{y, x z\right\} \right),$ which is equivalent to the Leibniz rule. For the second relation in (3), when applied to $z \in \mathcal{P}$, we obtain 
    $\left\{xy, z\right\} + (-1)^{xy}\left\{y, xz\right\} = 2 x\left\{y, z\right\}$.    
    The relation (4) is a consequence of the linearity of $Q$ and the identity (\ref{propeq1}).
\end{proof}

Additionally, we have the next relations involving the left multiplication operators of a transposed Poisson superalgebra.

\begin{proposition}\label{rels2}
Given $\mathcal{P}$ a transposed Poisson superalgebra and $x,y,z\in \mathcal{P}$, we have the following relations:
    \begin{enumerate}
        \item $P_{x}Q_{y} - (-1)^{xy}P_{y}Q_{x}  = - P_{\left\{x,y\right\}} $.

        \item  $(-1)^{zy}Q_{z}P_{\left\{x, y\right\}} + (-1)^{xz}Q_{x}P_{\left\{y, z\right\}} + (-1)^{yx}Q_{y}P_{\left\{z, x\right\}} = 0$.

        \item { $(-1)^{xz}Q_{\left\{x, y\right\}}P_{z} + (-1)^{yx}Q_{\left\{y, z\right\}}P_{x} + (-1)^{zy}Q_{\left\{z, x\right\}}P_{y} = 0$.}

        \item {$(-1)^{xz}P_{\left\{x, y\right\}}Q_{z} + (-1)^{yx}P_{\left\{y, z\right\}}Q_{x} + (-1)^{zy}P_{\left\{z, x\right\}}Q_{y} = 0$.}

        \item $ Q_{xy}P_{z} - (-1)^{zx + zy}Q_{zx}P_{y} = 2 P_{x\left\{y, z\right\}} $ and $(-1)^{yx}Q_{yz}P_{x} + (-1)^{zy}Q_{zx}P_{y} 
        = 2 (-1)^{xz}P_{xy}Q_{z}$.

        \item $ (-1)^{zy+zx}P_{z}Q_{xy} -  (-1)^{yx}P_{y}Q_{xz} = P_{x\left\{y, z\right\}} $.
    \end{enumerate}
\end{proposition}
\begin{proof}
    These relations are obtained, respectively, by rewriting the identities 
of the Proposition \ref{propeqs} using the maps $P$ and $Q$.
\end{proof}

 The following Theorem establishes a connection between the class of transposed Poisson superalgebras and the class of Jordan superalgebras through the Kantor double construction. 
 
\begin{theorem}\label{jordanbracket}
If $\mathcal{P}$ is a transposed Poisson algebra, then the algebra $\mathfrak{J}(\mathcal{P})$ is a Jordan superalgebra.
\end{theorem}
\begin{proof}
We verify the supercommutativity first. Given $x, y \in \mathcal{P}_0 \cup \mathcal{P}_1$ and corresponding $x^s, y^s \in \mathcal{P}_0^s \cup \mathcal{P}_1^s$, by the definition of the multiplication we have that $x*y = xy = (-1)^{xy} yx = (-1)^{xy} y*x$, 
$$x^s * y = (xy)^s = (-1)^{xy}(yx)^s = (-1)^{(x+1)y} y * x^s , $$
$$x^s * y^s = (-1)^x\left\{x, y\right\} = - (-1)^{xy+x}\left\{y, x\right\} = (-1)^{xy+y+x+1} y^s * x^s = (-1)^{(x+1)(y+1)} y^s * x^s.$$
Since the parity of $x^s$ is $|x|+1$, we conclude that $\mathfrak{J}(\mathcal{P})$ is supercommutative.

\medskip

Now, we check the Jordan superidentity (\ref{sjord}). 
Observe that the left multiplication on the superalgebra $\mathfrak{J}(\mathcal{P})$ are the linear operators $L_a, L_{a^s}: \mathfrak{J}(\mathcal{P})\rightarrow\mathfrak{J}(\mathcal{P}) $  corresponding to the matrices
$$L_a \equiv \begin{pmatrix}
P_a & 0 \\
0 & (-1)^aP_a 
\end{pmatrix}, \quad L_{a^s} \equiv \begin{pmatrix}
0 & (-1)^aQ_a \\
P_a & 0 
\end{pmatrix}.$$
where $a\in \mathcal{P}_0\cup \mathcal{P}_1$. A straightforward verification shows that $|L_a|=|a|$ and $|L_{a^s}| = |a|+1$. Thus, by the relations (1-3) in Proposition \ref{rels}, we have the following relations between the multiplication operators in $\mathfrak{J}(\mathcal{P})$.

    \begin{equation}\label{com1}
    \begin{split}
            [L_x, L_y] & = \begin{pmatrix}
            P_x & 0 \\
            0 & (-1)^xP_x 
            \end{pmatrix} \begin{pmatrix}
            P_y & 0 \\
            0 & (-1)^yP_y 
            \end{pmatrix} - (-1)^{xy}\begin{pmatrix}
            P_y & 0 \\
            0 & (-1)^yP_y 
            \end{pmatrix} \begin{pmatrix}
            P_x & 0 \\
            0 & (-1)^xP_x 
            \end{pmatrix} \\
             & = \begin{pmatrix}
            [P_x, P_y] & 0 \\
            0 & (-1)^{x+y}[P_x, P_y] 
            \end{pmatrix} = 0.
    \end{split}
    \end{equation}

    \begin{equation}\label{com2}
    \begin{split}
            [L_x, L_{y^s}] &= \begin{pmatrix}
            P_x & 0 \\
            0 & (-1)^xP_x 
            \end{pmatrix} \begin{pmatrix}
            0 & (-1)^yQ_y \\
            P_y & 0 
            \end{pmatrix} - (-1)^{x(y+1)}\begin{pmatrix}
            0 & (-1)^yQ_y \\
            P_y & 0 
            \end{pmatrix} \begin{pmatrix}
            P_x & 0 \\
            0 &  (-1)^xP_x 
            \end{pmatrix}  \\
            &=\begin{pmatrix}
            0 & (-1)^{y}[P_x, Q_y] \\
            (-1)^x[P_x, P_y]  & 0 
            \end{pmatrix} = \begin{pmatrix}
            0 & \frac{1}{2}(-1)^y(Q_{xy} - (-1)^{xy}Q_{y}P_{x}) \\
            0 & 0 
            \end{pmatrix}.
    \end{split}
    \end{equation}

    \begin{equation}\label{com3}
    \begin{split}
            [L_{x^s}, L_{y^s}] &= \begin{pmatrix}
            0 & (-1)^xQ_x \\
            P_x & 0 
            \end{pmatrix} \begin{pmatrix}
            0 & (-1)^yQ_y \\
            P_y & 0 
            \end{pmatrix} - (-1)^{(x+1)(y+1)} \begin{pmatrix}
            0 & (-1)^yQ_y \\
            P_y & 0 
            \end{pmatrix} \begin{pmatrix}
            0 & (-1)^xQ_x \\
            P_x & 0 
            \end{pmatrix} \\
            &= \begin{pmatrix}
            (-1)^x(Q_{x}P_{y}+ (-1)^{xy} Q_{y}P_{x}) & 0 \\
            0 & (-1)^y(P_{x}Q_{y} + (-1)^{xy} P_{y}Q_{x}) 
            \end{pmatrix}.
    \end{split}
    \end{equation}

\bigskip

Proceed by considering the various cases that arise depending on the parity. 
\begin{enumerate}
    \item For $x,y,z \in \mathcal{P}_0\cup \mathcal{P}_1$. Then the Jordan superidentity (\ref{sjord}) is verified, as a consequence of (\ref{com1}).

\medskip

    \item For $x,y \in \mathcal{P}$ and $z^s\in \mathcal{P}^s$, we have 
    \begin{equation*} 
    \begin{split}
        &(-1)^{x(z+1)}[L_{x* y}, L_{z^s}] + (-1)^{yx}[L_{y* z^s}, L_x] + (-1)^{(z+1)y}[L_{z^s* x}, L_y] \\
        &= (-1)^{x(z+1)}[L_{xy}, L_{z^s}] + (-1)^{yx}[L_{(-1)^y(y z)^s}, L_x] + (-1)^{(z+1)y}[L_{(zx)^{s}}, L_y].        
    \end{split}
    \end{equation*}

    Applying (\ref{com2}), we can write the expression above as the sum of matrices
    \begin{equation*}
        \begin{split}
    (-1)^{xz+x}\begin{pmatrix}
            0 & \frac{1}{2}(-1)^z(Q_{xyz} - (-1)^{xz+ yz}Q_{z}P_{xy}) \\
            0 & 0 
    \end{pmatrix} 
    &- (-1)^{xz+y+x}\begin{pmatrix}
            0 & \frac{1}{2}(-1)^{y+z}(Q_{xy z} - (-1)^{xy + xz}Q_{y z}P_{x}) \\
            0 & 0 
            \end{pmatrix} \\
    &- (-1)^{yx}\begin{pmatrix}
            0 & \frac{1}{2}(-1)^{x+z}(Q_{yzx} - (-1)^{yz+yx}Q_{zx}P_{y}) \\
            0 & 0 
            \end{pmatrix}.
        \end{split}
    \end{equation*}
    The upper right term, up to $\frac{1}{2}(-1)^{x+z}$, is equal to
    $$(-1)^{xz}(Q_{xyz} - (-1)^{xz+ yz}Q_{z}P_{xy}) - (-1)^{xz}(Q_{xy z} - (-1)^{xy + xz}Q_{y z}P_{x}) - (-1)^{yx}(Q_{yzx} - (-1)^{yz+yx}Q_{zx}P_{y}).$$
    
    After the simplification using the supercommutivity of the transposed Poisson superalgebra, we have
    $$-(-1)^{xz}Q_{xyz} - (-1)^{yz}Q_{z}P_{xy} + (-1)^{xy}Q_{y z}P_{x} +  (-1)^{yz}Q_{zx}P_{y}.$$

    Using the second relation in (3) of Proposition \ref{rels}, we obtain
    $$-2(-1)^{xz}P_{xy}Q_{z} + (-1)^{xy}Q_{y z}P_{x} +  (-1)^{yz}Q_{zx}P_{y}=0.$$
    
    Finally, this is zero by the equation (5) in Proposition \ref{rels2}.

\medskip

    \item For $x \in \mathcal{P}$ and $y^s, z^s\in \mathcal{P}^s$, we have
    \begin{equation*} 
    \begin{split}
        &(-1)^{x(z+1)}[L_{x* y^s}, L_{z^s}] + (-1)^{(y+1)x}[L_{y^s* z^s}, L_x] + (-1)^{(z+1)(y+1)} [L_{z^s* x}, L_{y^s}] \\
        &=  (-1)^{x(z+1)}[L_{(-1)^x(xy)^s}, L_{z^s}] + (-1)^{(y+1)x}[L_{(-1)^y\left\{y, z\right\}}, L_x] + (-1)^{(z+1)(y+1)} [L_{(zx)^s}, L_{y^s}]. 
    \end{split}
    \end{equation*}
    
    By the equations (\ref{com1}) and (\ref{com3}), the previous expression can be written as the following sum of matrices
    \begin{equation*}
        \begin{split}
            (-1)^{xz} \begin{pmatrix}
            (-1)^{x+y}(Q_{xy}P_{z}+ (-1)^{xz+yz} Q_{z}P_{xy}) & 0 \\
            0 & (-1)^z(P_{xy}Q_{z} + (-1)^{xz+yz} P_{z}Q_{xy}) 
    \end{pmatrix} + \\ (-1)^{zy+z+y+1} \begin{pmatrix}
            (-1)^{z+x}(Q_{zx}P_{y}+ (-1)^{zy+xy} Q_{y}P_{zx}) & 0 \\
            0 & (-1)^y(P_{zx}Q_{y} + (-1)^{zy+xy} P_{y}Q_{zx}) 
    \end{pmatrix}.
        \end{split}
    \end{equation*}

    The upper left term is, up to $(-1)^{x+y}$, equal to
    \begin{equation*}
        \begin{split}
           &(-1)^{xz}(Q_{xy}P_{z}+ (-1)^{xz+yz} Q_{z}P_{xy}) - (-1)^{zy} (Q_{zx}P_{y}+ (-1)^{zy+xy} Q_{y}P_{zx})\\
           &= (-1)^{xz}Q_{xy}P_{z}+ (-1)^{yz} Q_{z}P_{xy} - (-1)^{zy} Q_{zx}P_{y}- (-1)^{xy} Q_{y}P_{zx}.
        \end{split}
    \end{equation*}
    By the relation (5) in Proposition~\ref{rels2}, we have the Leibniz rule.
    $$(-1)^{xz}2P_{x\left\{y, z\right\}}+ (-1)^{yz} Q_{z}P_{xy} - (-1)^{xy} Q_{y}P_{zx}=0.$$
    
    The lower right term is, up to $(-1)^{z}$, equal to
    \begin{equation*}
        \begin{split}
          &(-1)^{xz}(P_{xy}Q_{z} + (-1)^{xz+yz} P_{z}Q_{xy}) - (-1)^{zy}(P_{zx}Q_{y} + (-1)^{zy+xy} P_{y}Q_{zx})\\
          &= (-1)^{xz}P_{xy}Q_{z} + (-1)^{yz} P_{z}Q_{xy} - (-1)^{zy}P_{zx}Q_{y} - (-1)^{xy} P_{y}Q_{zx}.
        \end{split}
    \end{equation*}
From here, we can use equation (6) in Proposition \ref{rels2} to obtain an expression which is zero, by (1) in Proposition \ref{rels2}.
    $$(-1)^{xz}P_{xy}Q_{z} - (-1)^{zy}P_{zx}Q_{y} + (-1)^{zx}P_{x\left\{y,z\right\}} = (-1)^{xz} P_{x}(P_{y}Q_{z} - (-1)^{zy}P_{z}Q_{y} + P_{\left\{y,z\right\}}) = 0.$$    

    %\item For $y \in \mathcal{P}$ and $x, z\in \mathcal{P}^s$. 

    %\begin{equation*}
    %    (-1)^{xz}[L_{x* y}, L_z] + (-1)^{yx}[L_{y* z}, L_x] + (-1)^{zy}[L_{z* x}, L_y] = 0
    %\end{equation*}

    %\item For $z \in \mathcal{P}$ and $x, y\in \mathcal{P}^s$.

    %\begin{equation*} 
    %    (-1)^{xz}[L_{x* y}, L_z] + (-1)^{yx}[L_{y* z}, L_x] + (-1)^{zy}[L_{z* x}, L_y] = 0
    %\end{equation*}

\medskip

    \item Lastly, for $x^s, y^s, z^s\in \mathcal{P}^s$, the Jordan identity is
    \begin{equation*} 
    \begin{split}
        & (-1)^{(x+1)(z+1)}[L_{x^s* y^s}, L_{z^s}]  + (-1)^{(y+1)(x+1)} [L_{y^s* z^s}, L_{x^s}] + (-1)^{(z+1)(y+1)} [L_{z^s* x^s}, L_{y^s}]\\ 
        &= (-1)^{(x+1)(z+1)}[L_{(-1)^x\left\{x, y\right\}}, L_{z^s}]  + (-1)^{(y+1)(x+1)} [L_{(-1)^y\left\{y, z\right\}}, L_{x^s}] + (-1)^{(z+1)(y+1)} [L_{(-1)^z\left\{z, x\right\}}, L_{y^s}]. \end{split}
    \end{equation*}

    Applying the equation (\ref{com2}) and arguing as before, every term in the operator matrix is zero except for %, up to $-\frac{1}{2}$, for 
    $$(-1)^{xz}Q_{\left\{x, y\right\}z} - (-1)^{yz}Q_{z}P_{\left\{x, y\right\}} + (-1)^{yx}Q_{\left\{y, z\right\}x} - (-1)^{zx}Q_{x}P_{\left\{y, z\right\}} + (-1)^{zy}Q_{\left\{z, x\right\}y} - (-1)^{xy}Q_{y}P_{\left\{z, x\right\}}.$$
    
    Using the relation (4) in Proposition \ref{rels}, we obtain the next expression, which is zero, by (2) in Proposition \ref{rels2}.
    $$- (-1)^{yz}Q_{z}P_{\left\{x, y\right\}} - (-1)^{zx}Q_{x}P_{\left\{y, z\right\}} - (-1)^{xy}Q_{y}P_{\left\{z, x\right\}} = 0.$$
\end{enumerate}

Due to the cyclic nature of the Jordan superidentity, the four cases considered here are enough to prove that $\mathfrak{J}(\mathcal{P})$ satisfies the Jordan superidentity (\ref{sjord}), and therefore, it is a Jordan superalgebra.
\end{proof}

Observe that if $\mathcal{P}$ has an ideal $\mathcal{I}$, then its Kantor double has an ideal $\mathcal{I}\oplus \mathcal{I}^s$. Therefore, if $\mathfrak{J}(\mathcal{P})$ is simple, then $\mathcal{P}$ is simple. The converse is more complicate. Kantor proved that the double of a non-trivial Poisson (super) algebra $\mathcal{P}$ is a simple Jordan superalgebra if and only if $\mathcal{P}$ is simple \cite{kantor92}. Later, King and McCrimmon extended this result to unital non-trivial Jordan brackets \cite{km95}. For non-unital Jordan brackets and, in particular, for non-unital transposed Poisson algebras, this result does not hold, as we can see in the next remark.

\begin{remark}
    Consider the family of simple transposed Poisson algebras ${\mathcal{P}}(\alpha,\beta)$ with $\alpha\beta=0$ over a field of characteristic $3$ of Remark \ref{example}.   The Kantor double is the Jordan superalgebra defined on ${\mathcal{P}}\oplus{\mathcal{P}}^{s}$, with multiplication given by
    \begin{equation*}
        \begin{split}
            &e_1*e_1 = \alpha e_2, \quad e_1 * e_1^s = \alpha e_2^s, \quad \quad  e_1^s*e_1 = \alpha e_2^s, \\
            &e_2*e_2 = \beta e_1, \quad e_2 * e_2^s = \beta e_1^s, \quad \quad  e_2^s*e_2 = \beta e_1^s, \\
            &e_1^s*e_2^s = e_3, \,\,\,\, \quad e_3^s*e_2^s = -2e_2, \,\, \quad e_3^s * e_1^s = 2 e_1.
        \end{split}
    \end{equation*}
    Observe that $\mathfrak{J}(\mathcal{P}) e_3 = e_3 \mathfrak{J}(\mathcal{P}) = 0$. Hence, the superalgebra $\mathfrak{J}(\mathcal{P})$ is not simple. {Also, note that $\mathcal{P}\mathcal{P}\neq \mathcal{P}$ and that the proper subspace $\mathcal{I} = \textrm{span}(\beta e_1, \alpha e_2)$ is a quasi-ideal of $\mathcal{P}$.}
\end{remark}

Moreover, suppose $\mathcal{P}$ is a simple transposed Poisson algebra such that $\mathcal{P}\mathcal{P}\neq \mathcal{P}$, then $\mathcal{P}\oplus (\mathcal{P}\mathcal{P})^s$ is an ideal of $\mathfrak{J}(\mathcal{P})$. Hence, if $\mathfrak{J}(\mathcal{P})$ is simple, then $\mathcal{P}$ is simple and $\mathcal{P}\mathcal{P}= \mathcal{P}$. Indeed, we have 
$$(\mathcal{P}\oplus (\mathcal{P}\mathcal{P})^s)*(\mathcal{P}\oplus \mathcal{P}^s)\subset \mathcal{P}\mathcal{P}+ (\mathcal{P}\mathcal{P})^s + ((\mathcal{P}\mathcal{P})\mathcal{P})^s + \left\{\mathcal{P}\mathcal{P}, \mathcal{P}\right\} \subset (\mathcal{P}\oplus (\mathcal{P}\mathcal{P})^s).$$

We want to investigate whether it is possible to obtain a simple Jordan superalgebra from a non-unital transposed Poisson algebra. The next key result about the quasi-ideals in a simple transposed Poisson algebra is obtained.

\begin{lemma}\label{perfectquasi}
    Any simple transposed Poisson (super) algebra $(\mathcal{P}, \circ, \left\{\cdot,\cdot\right\})$ such that $\mathcal{P}\mathcal{P} = \mathcal{P}$ contains no quasi-ideals.
\end{lemma}
\begin{proof}
    Suppose $\mathcal{I}$ is a quasi-ideal of $\mathcal{P}$. Choose a maximal subspace $\mathcal{I}'$ such that $\mathcal{I}'\mathcal{P} \subset \mathcal{I}$. Then $\mathcal{I}\subset \mathcal{I}'$ and $\mathcal{I}'\neq \mathcal{P}$, because $\mathcal{P}\mathcal{P}=\mathcal{P}$. Now, we have $\mathcal{I}'\mathcal{P} \subset \mathcal{I} \subset \mathcal{I}'$ and $\left\{\mathcal{I}, \mathcal{P}\right\}\subset \mathcal{I}'$, by the maximality of $\mathcal{I}'$. Moreover, using the transposed Leibniz rule, we have the following inclusions
    $$\left\{\mathcal{I}', \mathcal{P}\right\}= \left\{\mathcal{I}', \mathcal{P}\mathcal{P}\right\} \subset \mathcal{P}\left\{\mathcal{I}', \mathcal{P}\right\} + \left\{\mathcal{P}\mathcal{I}', \mathcal{P}\right\}.$$

    The first term of the sum above is $\mathcal{P}\left\{\mathcal{I}', \mathcal{P}\right\} = \mathcal{P}\mathcal{P}\left\{\mathcal{I}', \mathcal{P}\right\}\subset \left\{\mathcal{I}'\mathcal{P}, \mathcal{P}\mathcal{P}\right\}\subset \mathcal{I}'$, using equation (\ref{propeq5}). The second term is $\left\{\mathcal{P}\mathcal{I}', \mathcal{P}\right\}\subset \left\{\mathcal{I}, \mathcal{P}\right\}\subset \mathcal{I}'$. Therefore, the subspace $\mathcal{I}'$ is an ideal, which contradicts the simplicity.
\end{proof}

We can prove the simplicity criterion for the Kantor double of a transposed Poisson algebra. The converse was proved above.

\begin{theorem}
    Let $(\mathcal{P}, \circ, \left\{\cdot,\cdot\right\})$ be a simple transposed Poisson (super) algebra such that $\mathcal{P}\mathcal{P}=\mathcal{P}$, then $\mathfrak{J}(\mathcal{P})$ is simple.
\end{theorem}
\begin{proof}
    Suppose $\mathcal{I}$ is an ideal with projections $\mathcal{I}_0$, $\mathcal{I}_1^s$ on $\mathcal{P}$ and $\mathcal{P}^s$, respectively. By the definition of the Kantor double, we have the following relations
    $$\mathcal{P}\mathcal{I}_0\subset \mathcal{I}_0\cap \mathcal{I}_1, \quad \mathcal{P}\mathcal{I}_1\subset \mathcal{I}_1, \quad \left\{\mathcal{P}, \mathcal{I}_1\right\}\subset \mathcal{I}_0.$$

    The subspace $\mathcal{J} := \mathcal{I}_0\cap \mathcal{I}_1$ is a quasi-ideal of $\mathcal{P}$, since $\mathcal{J}\mathcal{P}\subset \mathcal{J}$ and $\left\{\mathcal{J}, \mathcal{P}\right\}\mathcal{P}\subset \mathcal{I}_0\mathcal{P}\subset \mathcal{J}$. 
    So either $\mathcal{J}=0$ or $\mathcal{J}=\mathcal{P}$, by Lemma \ref{perfectquasi}. 
    If $\mathcal{J} = 0$, then 
        $\mathcal{I}_0\mathcal{P}=0$ and, using the equation (\ref{propeq5}) we have $\mathcal{P}\left\{\mathcal{I}_0, \mathcal{P}\right\}= \mathcal{P}\mathcal{P}\left\{\mathcal{I}_0, \mathcal{P}\right\}\subset \left\{\mathcal{I}_0\mathcal{P}, \mathcal{P}\mathcal{P}\right\} \subset 0$, so $\mathcal{I}_0$ is a quasi ideal, and then either $\mathcal{I}_0 = 0$ or $\mathcal{I}_0 = \mathcal{P}$. In the first case, we have that  $\mathcal{I}_1$ is a non-zero ideal of $\mathcal{P}$, so $\mathcal{I}_1=\mathcal{P}$. But then $\mathcal{I}*\mathcal{P} = \mathcal{I}_1^s*\mathcal{P}^s=\left\{\mathcal{P}, \mathcal{P}\right\}\neq0$, which is a contradiction. In the second case, $\mathcal{I}_0=\mathcal{P}$ implies that $\mathcal{P}\mathcal{P}\subset \mathcal{J} = 0$, another contradiction. 
        If $\mathcal{J} = \mathcal{P}$, then the arguments are analogous to \cite[Theorem 3.4]{kantor92}.
\end{proof}

\begin{comment}
\red 
\begin{proposition}
        Let $(\mathcal{P}, \circ, \left\{\cdot,\cdot\right\})$ be a transposed Poisson algebra and suppose that the associated Lie bracket is perfect. A proper subspace $\mathcal{I}$ is a quasi-ideal if and only if $\left\{\mathcal{P}, \mathcal{I}\right\} \mathcal{P}\subset \mathcal{I}$.
\end{proposition}
\begin{proof}
By equation (\ref{propeq1}), we have $\mathcal{P}\mathcal{I} = \left\{\mathcal{P}, \mathcal{P}\right\}\mathcal{I}\subset \left\{\mathcal{I}, \mathcal{P}\right\}\mathcal{P} + \left\{\mathcal{P}, \mathcal{I}\right\}\mathcal{P} \subset \mathcal{I}$.    
\end{proof}
\black
\end{comment}

The following is an example of a family of simple infinite-dimensional Jordan superalgebras that arise from a non-Poisson transposed Poisson algebra, by considering the Kantor double. First, recall the notion of a mutation of an algebra.

\begin{definition}
    Let $(\mathcal{P}, \circ)$ be an associative commutative algebra and choose $q \in \mathcal{P}$. A mutation of the algebra $\mathcal{P}$ by the element $q$ is a new algebra $(\mathcal{P}, \circ_q)$, where for any $x, y \in \mathcal{P}$, we have the product
    $$x \cdot_q y = x\circ q \circ y.$$
\end{definition}

Let us construct the Jordan superalgebra that arises from the Laurent-Witt transposed Poisson algebra.

\begin{remark}
        Let us fix $\mathbb{F}$ to be the complex field. The Lie algebra of derivations of the algebra of Laurent Polynomials $\mathcal{L}$ is the Witt algebra $\mathcal{W}$. The Witt algebra is defined on the vector space generated by $\left\{e_i: i\in \mathbb{Z}\right\}$ and it is endowed with the multiplication $\left\{e_i, e_j\right\} = (i-j)e_{i+j}$. It is known that every transposed Poisson algebra constructed on the Witt algebra is a mutation of the algebra of Laurent polynomials \cite{FKL}. They are simple, none of these algebras is a Poisson algebra and they are unital if and only if the element in the mutation is invertible in $\mathcal{L}$. Recall that unital transposed Poisson algebras are contact brackets. Fix an element $q\in \mathcal{W}$. Let us denote by $\mathcal{P}_{\mathcal{W}}^{q}$ the transposed Poisson algebra consisting of the Witt algebra with the mutation of the Laurent polynomials corresponding to the element $q$. Then, the Kantor double is the vector space 
    $\mathfrak{J}(\mathcal{P}_{\mathcal{W}}^{q}) = \mathcal{W}\oplus \mathcal{W}^{s}$ with the multiplication:
    $$e_i*e_j = q e_{i+j}, \quad e_i^s*e_j = (q e_{i+j})^s, \quad e_i*e_j^s = (q e_{i+j})^s, \quad e_i^s*e_j^s = (i-j)e_{i+j}.$$    

         The associative commutative part of $\mathcal{P}_{\mathcal{W}}^{q}$ is perfect if and only if it is unital. So the Jordan superalgebras $\mathfrak{J}(\mathcal{P}_{\mathcal{W}}^{q})$ are simple if and only if $q$ is invertible in $\mathcal{L}$.
\end{remark}

In fact, they have to be unital to give rise to a simple Jordan superalgebra, as the following theorem shows.  Recall that an algebra  is differentiably simple for a family of derivations $\mathfrak{D}$ if it is non-trivial and it contains no $\mathfrak{D}$-invariant ideals. We say an algebra is differentiably simple if this is true for some family of derivations. Differentiably simple associative commutative algebras are unital \cite{posner}. 

\begin{theorem}\label{unital}
    Let $(\mathcal{P}, \circ, \left\{\cdot,\cdot\right\})$ be a simple transposed Poisson algebra  such that $\mathcal{P}\mathcal{P} = \mathcal{P}$, then $(\mathcal{P}, \circ)$ is unital.
\end{theorem}
\begin{proof}
    Let us prove that $(\mathcal{P}, \circ)$ is differentiably simple.
    Consider the set $\mathfrak{D}$ of linear endomorphisms of $\mathcal{P}$ given by 
    $$\mathfrak{D}= \left\{D_{xy}: x, y\in \mathcal{P}\right\},$$
    where $D_{xy}: \mathcal{P}\rightarrow \mathcal{P}$ is given by $D_{xy}(a) = \left\{ax, y\right\} - a\left\{x, y\right\}$. We have 

    \begin{equation*}
        \begin{split}
            D_{xy}(a)b + aD_{xy}(b) &= b \left\{ax, y\right\} - ab\left\{x, y\right\} + a\left\{bx, y\right\} - ab\left\{x, y\right\} \\ 
            &=\frac{1}{2}\left\{abx, y\right\} + \frac{1}{2}\left\{ax, by\right\} - ab\left\{x, y\right\} + \frac{1}{2}\left\{abx, y\right\} + \frac{1}{2}\left\{bx, ay\right\} - ab\left\{x, y\right\} \\ 
            &= \left\{abx, y\right\} - ab\left\{x, y\right\} = D_{xy}(ab).
        \end{split}
    \end{equation*}
    
    Therefore, $\mathfrak{D}$ is a family of derivations of the algebra $(\mathcal{P}, \circ)$. 
    Suppose there is some $\mathfrak{D}$-invariant ideal $\mathcal{I}$ of $(\mathcal{P}, \circ)$. Let us show that it is a quasi-ideal of $\mathcal{P}$. Indeed, using equation (\ref{propeq5}), we have
    $$\mathcal{P}\left\{\mathcal{I}, \mathcal{P}\right\} = \mathcal{P}\mathcal{P}\left\{\mathcal{I}, \mathcal{P}\right\} \subset \left\{\mathcal{I}\mathcal{P}, \mathcal{P}\mathcal{P}\right\}\subset  \left\{\mathcal{I}\mathcal{P}, \mathcal{P}\right\}.$$

    Since $\mathcal{I}$ is $\mathfrak{D}$-invariant, we have $\left\{ax, y\right\} - a\left\{x, y\right\} \in \mathcal{I}$, for any $x,y \in \mathcal{P}$ and $a\in \mathcal{I}$. That is, $\left\{ax, y\right\} \in \mathcal{I}$ and $\left\{\mathcal{I}\mathcal{P}, \mathcal{P}\right\} \subset \mathcal{I}$.
    This contradicts Lemma \ref{perfectquasi}, so $(\mathcal{P}, \circ)$ is differentiably simple. Hence, by \cite[Theorem 5]{posner}, it is unital.
\end{proof}

\black

\medskip

Let us close the paper with a special mention to the second functor introduced by Kantor in \cite{kantor92}.

\begin{remark} \label{examplelie}
    Kantor  defines a second construction that produces a Lie superalgebra,  given a Poisson algebra, by defining a new product $[\cdot, \cdot]$ on $\mathcal{P} \oplus \mathcal{P}^s$. Given $a, b \in \mathcal{P}$ and corresponding $a^s, b^s \in \mathcal{P}^s$, we set
\begin{equation*}
    [a, b] = \left\{a,b\right\}, \quad [a^s, b] = [a, b^s] = (\left\{a,b\right\})^s, \quad [a^s, b^s] = a\circ b.
\end{equation*}
However, a straightforward verification of the Jacobi superidentity for the complex $3$-dimensional transposed Poisson algebra ${\mathcal{P}}$ with a basis $e_1, e_2, e_3$  defined by the non-zero products:
$${\mathcal{P}}:\left\{ 
\begin{tabular}{l}
$ e_3 \cdot e_3 = e_1,$ \\ 
 $\left\{e_1, e_{3}\right\}=e_1+e_2,$ % $[e_2, e_3]= \alpha e_2,$
\end{tabular}%
\right. $$
%$\mathfrak{L}(e_3^s,e_3^s,e_3^s) = (-1)^{|e_3^s||e_3^s|}[e_3^s, e_3^s], e_3^s] + (-1)^{|e_3^s||e_3^s|}[[e_3^s, e_3^s], e_3^s] + (-1)^{|e_3^s||e_3^s|}[[e_3^s, e_3^s], e_3^s]=- 3[[e_3^s, e_3^s], e_3^s] = -3 [e_1, e_3^s] = -3(e_1+e_2)^s \neq 0$
shows that this ``double" is not always Lie in the transposed case. In fact, the Leibniz superidentity $$\mathfrak{L}(e_3^s,e_3^s,e_3^s)= -3(e_1+e_2)^s \neq 0.$$ 
\end{remark}

\end{document}